\documentclass{amsart}
\usepackage{mathptmx}
\usepackage{amssymb}
\usepackage{amsfonts}
\usepackage{amsmath}
\usepackage{graphicx}
\usepackage{shadow}
\usepackage{color}
\usepackage[all]{xy}
\usepackage[pagebackref]{hyperref}
\newtheorem{thm}{Theorem}[section]
\newtheorem{corollary}[thm]{Corollary}
\newtheorem{lemma}[thm]{Lemma}
\newtheorem{proposition}[thm]{Proposition}

\theoremstyle{definition}

\newtheorem{example}[thm]{Example}

\theoremstyle{remark}
\newtheorem{remark}[thm]{Remark}
\newtheorem{question}[thm]{Question}

\newtheorem{claim}{\sc Claim}
\newcommand{\field}[1]{\mathbb{#1}}

\newcommand{\Q }{\field{Q}}
\newcommand{\Z }{\field{Z}}

\DeclareMathOperator{\Nil}{Nil}
\DeclareMathOperator{\Ze}{Z}
\DeclareMathOperator{\J}{J}
\DeclareMathOperator{\U}{U}
\DeclareMathOperator{\Ann}{Ann}
\DeclareMathOperator{\Max}{Max}

\DeclareMathOperator{\Hom}{Hom}
\DeclareMathOperator{\Ker}{Ker}

\DeclareMathOperator{\Spec}{Spec}

\DeclareMathOperator{\w.dim}{w.gl.dim}
\DeclareMathOperator{\fd}{fd}
\def\m{\frak{m}}
\makeatletter
  \newcounter{xenumi}
  \newenvironment{xenumerate}{%
  \begin{list}{(\arabic{xenumi})}{
    \setcounter{xenumi}{1}\usecounter{xenumi}
    \setlength{\parsep}{4\p@ \@plus2\p@ \@minus\p@}
    \setlength{\topsep}{6\p@ \@plus2\p@ \@minus2\p@}
    \setlength{\itemsep}{2\p@ \@plus1\p@ \@minus\p@}
    \setlength{\labelwidth}{0mm}
    \setlength{\labelsep}{2mm}
    \setlength{\itemindent}{2mm}
    \setlength{\leftmargin}{0mm}
    \setlength{\listparindent}{0mm}
  }}{\end{list}}
\makeatother
\def\1{{\rm (1)}}
\def\2{{\rm (2)}}
\def\3{{\rm (3)}}
\def\4{{\rm (4)}}
\def\5{{\rm (5)}}
\def\6{{\rm (6)}}
\begin{document}
%%%%%%%%%%%%%%%%%%%%%%%%%%%%%%%%%%%%%%%%%%%%%%%%%%%%%%%%%%%%%%%%%%%%%%%%%%%%%%%%%%%%%%%%%%%%%%%%%%%%%%%%%%%%%%%%%%%%%%%%%%%%%%%%%%%%%%%%%%%%%%
%%%%%%%%%%%%%%%%%%%%%%%%%%%%%%%%%%%%%%%%%%%%%%%%%%%%%%%%%%%%%%%%%%%%%%%%%%%%%%%%%%%%%%%%%%%%%%%%%%%%%%%%%%%%%%%%%%%%%%%%%%%%%%%%%%%%%%%%%%%%%%
%%%%%%%%%%%%%%%%%%%%%%%%%%%%%%%%%%%%%%%%%%%%%%%%%%%%%%%%%%%%%%%%%%%%%%%%%%%%%%%%%%%%%%%%%%%%%%%%%%%%%%%%%%%%%%%%%%%%%%%%%%%%%%%%%%%%%%%%%%%%%%
%%%%%%%%%%%%%%%%%%%%%%%%%%%%%%%%%%%%%%%%%%%%%%%%%%%%%%%%%%%%%%%%%%%%%%%%%%%%%%%%%%%%%%%%%%%%%%%%%%%%%%%%%%%%%%%%%%%%%%%%%%%%%%%%%%%%%%%%%%%%%%
%%%%%%%%%%%%%%%%%%%%%%%%%%%%%%%%%%%%%%%%%%%%%%%%%%%%%%%%%%%%%%%%%%%%%%%%%%%%%%%%%%%%%%%%%%%%%%%%%%%%%%%%%%%%%%%%%%%%%%%%%%%%%%%%%%%%%%%%%%%%%%

\title[Pr\"ufer conditions in an amalgamated duplication]{Pr\"ufer conditions in an amalgamated duplication\\ of a ring along an ideal $^{(\star)}$}
\thanks{$^{(\star)}$ Supported by King Fahd University of Petroleum \& Minerals under Research Project \#: RG1208-1/2.}

\author{M. Chhiti}
\address{Department of Mathematics, FST, University S. M. Ben Abdellah, Fez 30000, Morocco}
\email{chhiti.med@hotmail.com}

\author{M. Jarrar}
\address{Department of Mathematics and Statistics, King Fahd University of Petroleum \& Minerals, Dhahran 31261, KSA}
\email{mojarrar@kfupm.edu.sa}

\author[S. Kabbaj]{S. Kabbaj $^{(1)}$}\thanks{$^{(1)}$ Corresponding author.}
\address{Department of Mathematics and Statistics, King Fahd University of Petroleum \& Minerals, Dhahran 31261, KSA}
\email{kabbaj@kfupm.edu.sa}

\author{N. Mahdou}
\address{Department of Mathematics, FST, University S. M. Ben Abdellah, Fez 30000, Morocco}
\email{mahdou@hotmail.com}

\date{\today}



\dedicatory{to Marco Fontana for his 65th Birthday}

\begin{abstract}
This paper investigates ideal-theoretic as well as homological extensions of the Pr\"ufer domain concept to commutative rings with zero divisors in an amalgamated duplication of a ring along an ideal. The new results both compare and contrast with recent results on trivial ring extensions (and pullbacks) as well as yield original families of examples issued from amalgamated duplications subject to various Pr\"ufer conditions.
\smallskip
\end{abstract}
\maketitle

%%%%%%%%%%%%%%%%%%%%%%%%%%%%%%%%%%%%%%%%%%%%%%%%%%%%%%%%%%%%%%%%%%%%%%%%%%%%%%%%%%%%%%%%%%%%%%%%%%%%%%%%%%%%%%%%%%%%%%%%%%%%%%%%%%%%%%%%%%%%%%
%%%%%%%%%%%%%%%%%%%%%%%%%%%%%%%%%%%%%%%%%%%%%%%%%%%%%%%%%%%%%%%%%%%%%%%%%%%%%%%%%%%%%%%%%%%%%%%%%%%%%%%%%%%%%%%%%%%%%%%%%%%%%%%%%%%%%%%%%%%%%%
%%%%%%%%%%%%%%%%%%%%%%%%%%%%%%%%%%%%%%%%%%%%%%%%%%%%%%%%%%%%%%%%%%%%%%%%%%%%%%%%%%%%%%%%%%%%%%%%%%%%%%%%%%%%%%%%%%%%%%%%%%%%%%%%%%%%%%%%%%%%%%
%%%%%%%%%%%%%%%%%%%%%%%%%%%%%%%%%%%%%%%%%%%%%%%%%%%%%%%%%%%%%%%%%%%%%%%%%%%%%%%%%%%%%%%%%%%%%%%%%%%%%%%%%%%%%%%%%%%%%%%%%%%%%%%%%%%%%%%%%%%%%%
%%%%%%%%%%%%%%%%%%%%%%%%%%%%%%%%%%%%%%%%%%%%%%%%%%%%%%%%%%%%%%%%%%%%%%%%%%%%%%%%%%%%%%%%%%%%%%%%%%%%%%%%%%%%%%%%%%%%%%%%%%%%%%%%%%%%%%%%%%%%%%
\section{Introduction}
\noindent All rings considered in this paper are commutative with unity and all modules are unital. Let $A$ be a ring, $I$ an ideal of $A$, and $\pi:A\rightarrow \frac{A}{I}$ the canonical surjection. The amalgamated duplication of $A$ along $I$, denoted by $A\bowtie I$, is the special pullback (or fiber product) of $\pi$ and $\pi$; i.e., the subring  of $A \times A$ given by $$A\bowtie I:=\pi\times_{\frac{A}{I}}\pi=\{(a,a+i)\mid a\in A, i\in I\}.$$
This construction was introduced and its basic properties were studied by D'Anna and Fontana in \cite{DF1,DF2} and then it was investigated by D'Anna in \cite{D} with the aim of applying it to curve singularities (over algebraic closed fields) where he proved that the amalgamated duplication of an algebroid curve along a regular canonical ideal yields a Gorenstein algebroid curve
\cite[Theorem 14 and Corollary 17]{D}. In \cite{DFF1,DFF2}, with Finocchiaro, they have considered the more general context of amalgamated algebra $A\bowtie^{f} J:=\{(a,f(a)+j)\mid a\in A, j\in J\}$ for a given homomorphism of rings $f: A\rightarrow B$ and ideal $J$ of $B$. In particular, they have studied amalgamations in the frame of pullbacks which allowed them to establish numerous (prime) ideal and ring-theoretic basic properties for this new construction. Two more recent works on amalgamated duplications are \cite{MY,Sh}. The interest of amalgamation resides, partly, in its ability to cover several basic constructions in commutative algebra, including pullbacks and trivial ring extensions (also called Nagata's idealizations).

A domain is Pr\"ufer if all its non-zero finitely generated ideals are invertible \cite{K,P}. There are well-known extensions of this notion to arbitrary rings (with zero divisors). Namely, for a ring $R$,\\
\1  $R$ is semihereditary, i.e., every finitely generated ideal of $R$ is projective \cite{CE};\\
\2  $R$ has weak global dimension $\w.dim(R)\leq1$ \cite{G1,G2};\\
\3  $R$ is arithmetical, i.e., every finitely generated ideal of $R$ is locally principal \cite{Fu,J};\\
\4  $R$ is an fqp-ring, i.e., every finitely generated ideal of $R$ is quasi-projective \cite{AJK}.\\
\5 $R$ is Gaussian, i.e., $c(fg)=c(f)c(g), \forall f,g\in R[x]$, where $c(f)$ is the content of $f$ \cite{T};\\
\6 $R$ is Pr\"ufer, i.e., every finitely generated regular ideal of $R$ is projective \cite{BS,Gr}.\\
The following diagram  summarizes the relations between these Pr\"ufer-like conditions where the implications cannot
be reversed in general \cite{AJK,BG,BG2,G2,G3}:
%%%%%%%%%%%%%%%%%%%%%%%%%%%%%%%%%%%%%%%%%%%%%%%%
\begin{center}
$R$ is semihereditary\\
$\Downarrow$\\
$\w.dim(R)\leq 1$\\
$\Downarrow$\\
$R$ is arithmetical\\
$\Downarrow$ \\
$R$ is an fqp-ring\\
$\Downarrow$\\
$R$ is Gaussian\\
$\Downarrow$\\
$R$ is Pr\"ufer
\end{center}
%%%%%%%%%%%%%%%%%%%%%%%%%%%%%%%%%%%%%%%%%%%%%%%%
All these forms coincide in the context of domains \cite{AJK,G3}. Glaz \cite{G3} and Bazzoni \& Glaz \cite{BG2} constructed examples which show that all these notions are distinct in the context of arbitrary rings. It is notable that original examples, marking the distinction of each of the above classes of Pr\"ufer-like rings are rare in the literature. New examples, in this regard, were provided via the study of these notions in diverse settings of trivial ring extensions \cite{AJK,BKM}. Pullbacks issued from rings with zero divisors were also considered for a similar study; namely, let $T$ be an arbitrary ring (possibly, with zero divisors), $I$ a (regular) ideal of $T$, $\pi:T\rightarrow \frac{T}{I}$ the canonical surjection, and $i:D\hookrightarrow \frac{T}{I}$ an inclusion of rings. Let $R:=i\times_{\frac{T}{I}}\pi$ be the pullback of $i$ and $\pi$. In \cite{Bo1,Bo2}, the author examined the transfer of the Pr\"ufer conditions (except the fqp property) from $D$ and $T$ to $R$. At this point, it is worthwhile noticing that an amalgamated duplication along an ideal $I$ collapses to a trivial ring extension $A\ltimes I$ for $I^{2}=0$ and overlaps with the above pullbacks for $I=0$ (i.e., $A\bowtie I\cong A$).

This paper investigates necessary and sufficient conditions for an amalgamated duplication of a ring along an ideal to inherit the six aforementioned Pr\"ufer notions, and hence provides new families of examples subject to these conditions. In this vein, we shall omit the case $A\bowtie A=A\times A$ since all these notions are stable under finite products by \cite[Theorem 3.4]{B} and Remark~\ref{fqp1}. That is, in all main results, the ideal $I$ of the amalgamation will be assumed to be proper. Section 2 examines the transfer of the notions of local Pr\"ufer ring and total ring of quotients. Section 3  deals with the arithmetical, Gaussian, and fqp conditions. Section 4 is devoted to the weak global dimension and the transfer of the semihereditary condition.

Throughout,  $A\bowtie I$ will denote the amalgamated duplication of a ring $A$ along an ideal $I$ of $A$. If $J$ is an ideal of $A$, then
$J\bowtie I:=\{(j,j+i)\mid j\in J, i\in I\}$ is an ideal of $A\bowtie I$ with $\frac{A\bowtie I}{J\bowtie I}\cong \frac{A}{J}$ \cite[Proposition 5.1]{DFF1}. Under the natural injection $A\hookrightarrow A\bowtie I$ defined by $i(a)=(a,a)$, we identify $A$ with its respective image in $A\bowtie I$; and the natural surjection $A\bowtie I\twoheadrightarrow A$ yields the isomorphism
$\frac{A\bowtie I}{(0)\bowtie I}\cong A$ \cite[Remark 1]{D}. Also, for a ring $R$, $Q(R)$ will denote the total ring of quotients and $\Ze(R)$, $\U(R)$, $\Nil(R)$, and $\J(R)$ will denote, respectively, the set of zero divisors, set of invertible elements, nilradical, and Jacobson radical of $R$. Finally, $\Max(R)$ shall denote the set of maximal ideals of $R$, $\Max(R, I):=\{\m\in\Max(R)\mid I\subseteq \m\}$, and $\Ann(I)$ the annihilator of $I$ for any ideal $I$ of $R$.

%%%%%%%%%%%%%%%%%%%%%%%%%%%%%%%%%%%%%%%%%%%%%%%%%%%%%%%%%%%%%%%%%%%%%%%%%%%%%%%%%%%%%%%%%%%%%%%%%%%%%%%%%%%%%%%%%%%%%%%%%%%%%%%%%%%%%%%%%%%%%%
%%%%%%%%%%%%%%%%%%%%%%%%%%%%%%%%%%%%%%%%%%%%%%%%%%%%%%%%%%%%%%%%%%%%%%%%%%%%%%%%%%%%%%%%%%%%%%%%%%%%%%%%%%%%%%%%%%%%%%%%%%%%%%%%%%%%%%%%%%%%%%
%%%%%%%%%%%%%%%%%%%%%%%%%%%%%%%%%%%%%%%%%%%%%%%%%%%%%%%%%%%%%%%%%%%%%%%%%%%%%%%%%%%%%%%%%%%%%%%%%%%%%%%%%%%%%%%%%%%%%%%%%%%%%%%%%%%%%%%%%%%%%%
%%%%%%%%%%%%%%%%%%%%%%%%%%%%%%%%%%%%%%%%%%%%%%%%%%%%%%%%%%%%%%%%%%%%%%%%%%%%%%%%%%%%%%%%%%%%%%%%%%%%%%%%%%%%%%%%%%%%%%%%%%%%%%%%%%%%%%%%%%%%%%
%%%%%%%%%%%%%%%%%%%%%%%%%%%%%%%%%%%%%%%%%%%%%%%%%%%%%%%%%%%%%%%%%%%%%%%%%%%%%%%%%%%%%%%%%%%%%%%%%%%%%%%%%%%%%%%%%%%%%%%%%%%%%%%%%%%%%%%%%%%%%%
\section{Transfer of the Pr\"ufer condition}\label{P}

This section handles the notion of Pr\"ufer ring. An ideal $I$ of a ring $R$ is invertible if $II^{-1}=R$, where $I^{-1}:=\{x\in Q(R)\mid xI\subseteq R\}$; and $R$ is Pr\"ufer if every finitely generated regular ideal of $R$ is invertible (or, equivalently, projective) \cite{BS,Gr}. We refer the reader to \cite[Theorem 2.13]{BG} which collects fifteen conditions equivalent to this definition. Finally, recall that the class of Pr\"ufer rings contains strictly the class of total rings of quotients.

Next, before we announce the main result of this section (Theorem~\ref{P1}), we make the following useful remark.

%%%%%%%%%%%%%%%%%%%%%%%%%%%%%%%%%%%%%%%%%%%%%%%%%%%%%%%%%%%%%%%%%%%%%%
%%%%%%%%%%%%%%%%%%%%%%%%%%%%%%%%%%%%%%%%%%%%%%%%%%%%%%%%%%%%%%%%%%%%%%
\begin{remark}\label{P5}
Let $A$ be a ring, $I$ an ideal of $A$, and $P$ a prime ideal of $A$. In \cite[Propositions 5 \& 7]{D}, D'Anna proved that if  $I\nsubseteqq P$, then $\tilde{P}:=\{(p+i,p)\mid p\in P,\ i\in I\}$ and $P\bowtie I$ are the only prime ideals of  $A\bowtie I$ lying over $P$ and we have $$\dfrac{A\bowtie I}{\tilde{P}}\cong\dfrac{A\bowtie I}{P\bowtie I}\cong \dfrac{A}{P}\ \mbox{ and }\
    (A\bowtie I)_{\tilde{P}}\cong(A\bowtie I)_{P\bowtie I}\cong A_{P}.$$
Notice that $P\bowtie I$ and $\tilde{P}$ are incomparable. However, if $I\subseteq P$, then $P\bowtie I=\tilde{P}$ is the unique prime ideal of $A\bowtie I$ lying over $P$ and we have
$$\dfrac{A\bowtie I}{P\bowtie I}\cong \dfrac{A}{P}\ \mbox{ and }\
    (A\bowtie I)_{P\bowtie I}\cong A_{P}\bowtie I_{P}.$$
As a consequence, $(A,\m)$ is local with $I\subseteq \m$ if and only if $A\bowtie I$ is local with maximal ideal $\m\bowtie I$. This basic fact will be used throughout this paper without explicit mention.
\end{remark}

Now, to the main result:

%%%%%%%%%%%%%%%%%%%%%%%%%%%%%%%%%%%%%%%%%%%%%%%%%%%%%%%%%%%%%%%%%%%%%%
%%%%%%%%%%%%%%%%%%%%%%%%%%%%%%%%%%%%%%%%%%%%%%%%%%%%%%%%%%%%%%%%%%%%%%
\begin{thm}\label{P1}
Let $(A, \m)$ be a local ring and $I$ a proper ideal of $A$. Then $A\bowtie I$ is a Pr\"ufer ring if and only if $A$ is a Pr\"ufer ring and $I=aI$ for every $a\in \m\setminus \Ze(A)$.
\end{thm}

%%%%%%%%%%%%%%%%%%%%%%%%%%%%%%%%%%%%%%%%%%%%%%%%%%%%%%%%%%%%%%%%%%%%%%
%%%%%%%%%%%%%%%%%%%%%%%%%%%%%%%%%%%%%%%%%%%%%%%%%%%%%%%%%%%%%%%%%%%%%%
\begin{remark}\label{P5.1}
\begin{xenumerate}
\item Let $(A,\m)$ be a local Pr\"ufer ring. One can easily check that:
\begin{center}
$aI=a^{2}I,\ \forall\ a\in \m$ (i.e., $\forall\ a\in A$)\\
$\Downarrow$\\
$I=aI,\ \forall\ a\in \m\setminus \Ze(A)$ (i.e., $\forall\ a\in A\setminus \Ze(A)$)\\
$\Downarrow$\\
$I\subseteq \Ze(A)\subseteq \m$.
\end{center}
So, by Theorem~\ref{P1}, if $I$ is a proper regular ideal of $A$ (i.e., $I\nsubseteq\Ze(A)$), then the amalgamation  $A\bowtie I$ is never a Pr\"ufer ring. The first assumption ``$aI=a^{2}I,\ \forall\ a\in \m$" will be used later in Theorem~\ref{AGfqp1} to characterize amalgamations subject to the Gaussian and fqp conditions.

\item Notice that, in the setting $0\not=I\subseteq \Ze(A)$, the assumption ``$ I=aI,\ \forall\ a\in \m\setminus \Ze(A)$" is not necessarily embedded in the (local) Pr\"ufer condition. For instance, let $A:=\Z_{(2)}\ltimes \Q$ and $I:=0\ltimes \Z_{(2)}$. Then $A$ is a chained ring \cite[Theorem 2.1(2)]{BKM} with maximal ideal $2\Z_{(2)}\ltimes \Q$ and $I^{2}=0$; whereas $I\not= (4,0)I$. So, by Theorem~\ref{P1}, $A\bowtie I = A\ltimes I$ is not a Pr\"ufer ring.
\end{xenumerate}
\end{remark}

The proof of the theorem relies on the following lemmas which are of independent interest. Recall at this point that a polynomial $f$ over a ring $R$ is Gaussian if the content ideal equation $c(fg) = c(f)c(g)$ holds for any polynomial $g$ over $R$ \cite{T}.

%%%%%%%%%%%%%%%%%%%%%%%%%%%%%%%%%%%%%%%%%%%%%%%%%%%%%%%%%%%%%%%%%%%%%%
%%%%%%%%%%%%%%%%%%%%%%%%%%%%%%%%%%%%%%%%%%%%%%%%%%%%%%%%%%%%%%%%%%%%%%
\begin{lemma}\label{P2} Let $A$ be a ring and $I$ an ideal of
$A$. If the polynomial $F(x):=\sum_{i=0}^{n}(a_{i},a_{i})x^{i}$ is
Gaussian over $A\bowtie I$, then $f(x):=\sum_{i=0}^{n}a_{i}x^{i}$
is Gaussian over $A$.
\end{lemma}

%%%%%%%%%%%%%%%%%
%%%%%%%%%%%%%%%%%
\begin{proof}
Straightforward.
\end{proof}

%%%%%%%%%%%%%%%%%%%%%%%%%%%%%%%%%%%%%%%%%%%%%%%%%%%%%%%%%%%%%%%%%%%%%%
%%%%%%%%%%%%%%%%%%%%%%%%%%%%%%%%%%%%%%%%%%%%%%%%%%%%%%%%%%%%%%%%%%%%%%
\begin{lemma}\label{P3}
Let $A$ be a ring and $I$ an ideal of $A$. If $A\bowtie I$ is Pr\"ufer, then $A$ is Pr\"ufer.
\end{lemma}

%%%%%%%%%%%%%%%%%
%%%%%%%%%%%%%%%%%
\begin{proof}
Assume that $A\bowtie I$ is a Pr\"ufer ring. Let $J:=\sum_{i=0}^{n}a_{i}A$ be a finitely generated regular ideal of $A$ and $a$ a regular element of $J$.  Clearly, $G:=\sum_{i=0}^{n}(a_{i},a_{i})A\bowtie I$ is a finitely generated regular ideal of $A\bowtie I$ since $(a,a)\in G$. Then $G$ is invertible and hence the polynomial $F(x):=\sum_{i=0}^{n}(a_{i},a_{i})x^{i}$ is Gaussian over $A\bowtie
I $. By Lemma~\ref{P2}, $f(x):=\sum_{i=0}^{n}a_{i}x^{i}$ is Gaussian over $A$. Therefore
$J=c(f)$ is invertible in $A$ by \cite[Theorem 4.2]{BG}, making $A$ a Pr\"ufer ring, as desired.
\end{proof}

Next we recall a nice result by Maimani and Yassemi which provides a full description for the set of zero divisors of $A\bowtie I$ for any arbitrary commutative ring. In the sequel, the subset $\{(a,a+i)\mid a\in\Ze(A), i\in I\}$ of $A\bowtie I$ will be denoted by $\Ze(A)\bowtie I$.

%%%%%%%%%%%%%%%%%%%%%%%%%%%%%%%%%%%%%%%%%%%%%%%%%%%%%%%%%%%%%%%%%%%%%%
%%%%%%%%%%%%%%%%%%%%%%%%%%%%%%%%%%%%%%%%%%%%%%%%%%%%%%%%%%%%%%%%%%%%%%
\begin{lemma}[{\cite[Proposition 2.2]{MY}}]\label{P3.1}
Let $A$ be a ring and $I$ an ideal of $A$. Then
\[\begin{array}{rl}
\Ze(A\bowtie I) &=\Ze(A)\bowtie I\ \cup\ \big\{(i,0)\mid i\in I\big\}\\
                &\cup\ \big\{(a, a+i)\mid a\ \mbox{regular and}\ j(a+i)=0\ \mbox{for some}\ 0\not=j\in I\big\}.
                \end{array}\]
\end{lemma}

%%%%%%%%%%%%%%%%%%%%%%%%%%%%%%%%%%%%%%%%%%%%%%%%%%%%%%%%%%%%%%%%%%%%%%
%%%%%%%%%%%%%%%%%%%%%%%%%%%%%%%%%%%%%%%%%%%%%%%%%%%%%%%%%%%%%%%%%%%%%%
\begin{lemma}\label{P3.2}
Let $R$ be a local Pr\"ufer ring and let $x$ be a regular element of $R$. Then $xR$ is comparable with every principal ideal of $R$.
\end{lemma}

%%%%%%%%%%%%%%%%%
%%%%%%%%%%%%%%%%%
\begin{proof}
The proof follows immediately from \cite[Lemma 3.8]{AJK}.
\end{proof}

%%%%%%%%%%%%%%%%%%%%%%%%%%%%%%%%%%%%%%%%%%%%%%%%%%%%%%%%%%%%%%%%%%%%%%
%%%%%%%%%%%%%%%%%%%%%%%%%%%%%%%%%%%%%%%%%%%%%%%%%%%%%%%%%%%%%%%%%%%%%%
\begin{lemma}\label{P4}
Let $A$ be a local Pr\"ufer ring and $I$ an ideal of $A$. Then: $$I\subseteq \Ze(A) \Leftrightarrow\Ze(A\bowtie I)=\Ze(A)\bowtie I.$$
\end{lemma}

%%%%%%%%%%%%%%%%%
%%%%%%%%%%%%%%%%%
\begin{proof}
Assume $I\subseteq \Ze(A)$. Let $a$ be a regular element of $A$ and let $i\in I$. We claim that $a+i$ is regular in $A$. Indeed, the ideals $aA$ and $iA$ are comparable by Lemma~\ref{P3.2}. It follows that $i=ka$ for some non-unit $k\in A$ since $I\subseteq \Ze(A)$. Thus $a+i=(1+k)a\in A\setminus\Ze(A)$, as claimed. Consequently,
the set $\big\{(a, a+i)\mid a\ \mbox{regular and}\ j(a+i)=0\ \mbox{for some}\ 0\not=j\in I\big\}$ is empty. In view of the description of $\Ze(A\bowtie I)$ in Lemma~\ref{P3.1}, it merely collapses to $\Ze(A)\bowtie I$, as desired. The converse is trivial by the same lemma.
\end{proof}

%%%%%%%%%%%%%%%%%%%%%%%%%%%%%%%%%%%%%%%%%%%%%%%%%%%%%%%%%%%%%%%%%%%%%%%%%%%%%%%%%%%%%%%%%%%%
\begin{proof}[Proof of Theorem~\ref{P1}]
(1) $(A,\m)$ is assumed to be local and $I\subseteq \m$. This is equivalent to saying that $A\bowtie I$ is local. Suppose that $A\bowtie I$ is Pr\"ufer. By Lemma~\ref{P3}, $A$ is Pr\"ufer. Note that $\Ze(A)\subseteq\m$. We claim that $I\subseteq \Ze(A)$. Deny and let $i\in I\setminus \Ze(A)$. Clearly, $(i,i)$ is regular in $A\bowtie I$. By Lemma~\ref{P3.2}, the ideals $\big((0,i)\big)$ and $\big((i,i)\big)$ must be comparable in $A\bowtie I$ and, necessarily, $(0,i)=(i,i)(b,b+j)$ for some $b\in A$ and $j\in I$. So that $b=0$ and $i=ij$, whence $j=1$, the desired contradiction. Next, let $a\in A\setminus \Ze(A)$ and $i\in I$. By Lemma~\ref{P4}, $(a,a+i)$ is regular in $A\bowtie I$. As above, via Lemma~\ref{P3.2}, we get $(0,i)=(a,a+i)(b,b+j)$ for some $b\in A$ and $j\in I\subseteq \m$. Therefore, $b=0$ and thus $i=aj(1-j)^{-1}\in aI$, as desired.

Conversely, suppose $A$ is a (local) Pr\"ufer ring with $I=aI$ for every $a\in \m\setminus \Ze(A)$ (i.e., for every $a\in A\setminus \Ze(A)$). Let $F:=\big((a,a+i),(b,b+j)\big)$ be a regular ideal of $A\bowtie I$. Assume one, at least, of the two generators of $F$ is regular. By Lemma~\ref{P4}, $a$ or $b$ is regular in $A$. So, by Lemma~\ref{P3.2}, $(a)$ and $(b)$ are comparable in $A$; say, $a$ is regular in $A$ and $b=ac$ for some $c\in A$. By hypothesis, there is $k\in I$ such that $j-ic=(a+i)k$. So one can easily check that $(b,b+j)=(a,a+i)(c,c+k)$; i.e., $F:=\big((a,a+i)\big)$. Now, assume both generators of $F$ are zero divisors and let $(r, r+h)$ be a regular element of $F$. Similar arguments as above yield $a=ra'$, $i-ha'=(r+h)k_{1}$, $b=rb'$, and $j-hb'=(r+h)k_{2}$, for some $a',b'\in A$ and $k_{1}, k_{2}\in I$; leading to $F:=\big((r,r+h)\big)$. So in both cases $F$ is principal (and, a fortiori, invertible) making $A\bowtie I$ a Pr\"ufer ring \cite[Theorem 2.13(2)]{BG}. This completes the proof of the local case.
\end{proof}

As an application of Theorem~\ref{P1} (combined with Theorem~\ref{AGfqp1}), one can construct new examples of (non-Gaussian) Pr\"ufer rings as shown below.

%%%%%%%%%%%%%%%%%%%%%%%%%%%%%%%%%%%%%%%%%%%%%%%%%%%%%%%%%%%%%%%%%%%%%%
%%%%%%%%%%%%%%%%%%%%%%%%%%%%%%%%%%%%%%%%%%%%%%%%%%%%%%%%%%%%%%%%%%%%%%
\begin{example}\label{P5.2}
Let $R:=\frac{\Z}{8\Z}\bowtie \frac{2\Z}{8\Z}$. We have $\Ze(\frac{\Z}{8\Z})=\frac{2\Z}{8\Z}$ and hence, by Theorem~\ref{P1}, $R$  is a local Pr\"ufer ring (which is not Gaussian by Theorem~\ref{AGfqp1}(2)). Further, $R$ is neither a trivial ring extension nor a pullback of the type studied in \cite{Bo1,Bo2,Bo3}.
\end{example}

Total rings of quotients are important source of Pr\"ufer rings. Next, we study the transfer of this notion to an amalgamation.

%%%%%%%%%%%%%%%%%%%%%%%%%%%%%%%%%%%%%%%%%%%%%%%%%%%%%%%%%%%%%%%%%%%%%%
%%%%%%%%%%%%%%%%%%%%%%%%%%%%%%%%%%%%%%%%%%%%%%%%%%%%%%%%%%%%%%%%%%%%%%
\begin{proposition}\label{P6}
Let $A$ be a ring and $I$ an ideal of $A$ such that $I\subseteq \J(A)$. Then $A$ is a total ring of quotients if and only if $A\bowtie I$ is a total ring of quotients.
\end{proposition}

%%%%%%%%%%%%%%%%%
%%%%%%%%%%%%%%%%%
\begin{proof}
Assume $A$ is a total ring of quotients and let $(x,x+i)\in A\bowtie I$. If $x$ is a zero divisor in $A$, then so is (x,x+i) in $A\bowtie I$ since $\Ze(A)\bowtie I\subseteq \Ze(A\bowtie I)$ always holds. Now suppose that $x$ is invertible in $A$ and let $y:=x^{-1}$ and $j:=-iy^{2}(1+yi)^{-1}$. Since $I\subseteq \J(A)$, then $j\in I$. Further, we have $(x,x+i)(y,y+j)=(1,1)$. So $(x,x+i)$ is invertible in $A\bowtie I$. Conversely, assume $A\bowtie I$ is a total ring of quotients and let $x\in A$. Then $(x,x)$ is either a zero divisor or invertible in $A\bowtie I$. Clearly, this forces $x$ to be either a zero divisor or invertible in $A$, completing the proof.
\end{proof}

Let $(A,\m)$ be a local ring and let $n$ be an integer $\geq2$. By Proposition~\ref{P6},  $\frac{A}{\m^{n}}\bowtie \frac{\m^{n-1}}{\m^{n}}$ $\left(=\ \frac{A}{\m^{n}}\ltimes \frac{\m^{n-1}}{\m^{n}}\right)$ is a local total ring of quotients and, a fortiori, a local Pr\"ufer ring.

Recall that the notion of Pr\"ufer ring is not stable under factor rings \cite[Example 3.3]{BSh} (also \cite[Example 3.6]{L} and \cite[Example 2.8]{BKM}). A ring $R$ is locally Pr\"ufer if $R_{p}$ is Pr\"ufer $\forall\ p\in\Spec(R)$ \cite[Definition 2.1]{Bo3}. Lucas proved that if $R_{\m}$ is Pr\"ufer $\forall\ \m\in\Max(R)$ (a fortiori, if $R$ is locally Pr\"ufer), then $R$ is Pr\"ufer \cite[Proposition 2.10]{L}; and constructed a non-local Pr\"ufer ring which is not locally Pr\"ufer \cite[Example 2.11]{L}. Recently, Boynton provided an example of a local Pr\"ufer ring which is not locally Pr\"ufer \cite[Example 2.4]{Bo3}.

%%%%%%%%%%%%%%%%%%%%%%%%%%%%%%%%%%%%%%%%%%%%%%%%%%%%%%%%%%%%%%%%%%%%%%
%%%%%%%%%%%%%%%%%%%%%%%%%%%%%%%%%%%%%%%%%%%%%%%%%%%%%%%%%%%%%%%%%%%%%%
\begin{question}
Is  Theorem~\ref{P1} valid in the global case? i.e., when $A$ is Pr\"ufer (not necessarily local) or locally Pr\"ufer. One, particularly, needs to find the \emph{right and natural} globalization for the assumption ``$I=aI,\ \forall\ a\in \m\setminus \Ze(A)$."
\end{question}

%%%%%%%%%%%%%%%%%%%%%%%%%%%%%%%%%%%%%%%%%%%%%%%%%%%%%%%%%%%%%%%%%%%%%%%%%%%%%%%%%%%%%%%%%%%%%%%%%%%%%%%%%%%%%%%%%%%%%%%%%%%%%%%%%%%%%%%%%%%%%%
%%%%%%%%%%%%%%%%%%%%%%%%%%%%%%%%%%%%%%%%%%%%%%%%%%%%%%%%%%%%%%%%%%%%%%%%%%%%%%%%%%%%%%%%%%%%%%%%%%%%%%%%%%%%%%%%%%%%%%%%%%%%%%%%%%%%%%%%%%%%%%
%%%%%%%%%%%%%%%%%%%%%%%%%%%%%%%%%%%%%%%%%%%%%%%%%%%%%%%%%%%%%%%%%%%%%%%%%%%%%%%%%%%%%%%%%%%%%%%%%%%%%%%%%%%%%%%%%%%%%%%%%%%%%%%%%%%%%%%%%%%%%%
%%%%%%%%%%%%%%%%%%%%%%%%%%%%%%%%%%%%%%%%%%%%%%%%%%%%%%%%%%%%%%%%%%%%%%%%%%%%%%%%%%%%%%%%%%%%%%%%%%%%%%%%%%%%%%%%%%%%%%%%%%%%%%%%%%%%%%%%%%%%%%
% remember to contrast/compare with Kabbour/Mahdou results in thesis/paper on the arithmetical notion
%%%%%%%%%%%%%%%%%%%%%%%%%%%%%%%%%%%%%%%%%%%%%%%%%%%%%%%%%%%%%%%%%%%%%%%%%%%%%%%%%%%%%%%%%%%%%%%%%%%%%%%%%%%%%%%%%%%%%%%%%%%%%%%%%%%%%%%%%%%%%%
\section{Transfer of the arithmetical, Gaussian, and fqp conditions}\label{AGfqp}

A ring $R$ is arithmetical if the ideals of any localization of $R$ are linearly ordered; equivalently, if every finitely generated ideal of $R$ is locally principal \cite{Fu,J}. A local arithmetical ring is also called a chained or valuation ring. The ring $R$ is Gaussian if for every $f, g$ in the polynomial ring $R[x]$, one has the content ideal equation $c(fg) = c(f)c(g)$ \cite{T}. Both arithmetical and Gaussian  notions are local; i.e., a ring is arithmetical (resp., Gaussian) if and only if its localizations with respect to maximal ideals are arithmetical (resp., Gaussian). We will make frequent use of an important characterization of a local Gaussian ring; namely, ``\emph{for any two elements $a, b$ in the ring, we have $(a,b)^{2}=(a^{2})$ or $(b^{2})$; moreover, if $ab=0$ and, say, $(a,b)^{2}=(a^{2})$, then $b^{2}=0$}" \cite[Theorem 2.2]{BG2}.

An ideal $I$ of a ring $R$ is quasi-projective if the natural map $\Hom_{R}(I,I)\rightarrow\Hom_{R}(I,I/J)$, defined by $f\mapsto \overline{f}$, is surjective for every subideal $J$ of $I$. A ring $R$ is an fqp-ring if every finitely generated ideal of $R$ is quasi-projective \cite{AJK}. An arithmetical ring is an fqp-ring and an fqp-ring is Gaussian, where the implications are irreversible in general \cite[Theorem 3.2]{AJK}. It is worthwhile recalling, at this point, that the fqp condition is stable under formation of rings of fractions \cite[Lemma 3.6]{AJK}; though, the question of whether it is a local property is still elusively open \cite{AJK}. As mentioned in the introduction, we shall omit the case $A\bowtie A=A\times A$ since the fqp property, too, is stable under finite products as shown below.

%%%%%%%%%%%%%%%%%%%%%%%%%%%%%%%%%%%%%%%%%%%%%%%%%%%%%%%%%%%%%%%%%%%%%%
%%%%%%%%%%%%%%%%%%%%%%%%%%%%%%%%%%%%%%%%%%%%%%%%%%%%%%%%%%%%%%%%%%%%%%
\begin{remark}\label{fqp1}
Let  $R_{1}$ and $R_{2}$ be two fqp-rings, $R:=R_{1}\times R_{2}$, and $I:=I_{1}\times I_{2}$, where $I_{i}$ is a finitely generated ideal of $R_{i}$
for $i=1,2$. Let $f: I\rightarrow I/K$ be an $R$-map, where $K$ is a subideal of $I$ and write $K=K_{1}\times K_{2}$ and  $f=f_{1}\times f_{2}$, where $K_{i}$  is a subideal of $I_{i}$ and $f_{i}\in\Hom_{R}(I_{i},I_{i}/K_{i})$ defined by $f_{1}(x):=a$ such that $f(x,0)=(a,b)$ and similarly for $f_{2}$. Therefore, there is $g_{i}\in\Hom_{R}(I_{i},I_{i})$ such that $\overline{g_{i}}=f_{i}$. It is clear that $\overline{g}=\overline{g_{1}}\times \overline{g_{2}}=f $. It follows that $R$ is an fqp-ring. The converse is more straightforward.
\end{remark}

The main result of this section examines necessary and sufficient conditions for amalgamations issued from local rings to inherit the notions of arithmetical, Gaussian, and fqp-ring, respectively. In particular, it turns out that, among amalgamated duplications of local rings, only trivial extensions can inherit the Gaussian or fqp properties. Thereby, a second result examines the global case.

%%%%%%%%%%%%%%%%%%%%%%%%%%%%%%%%%%%%%%%%%%%%%%%%%%%%%%%%%%%%%%%%%%%%%%
% remember to contrast/compare with Kabbour/Mahdou results in thesis/paper on the arithmetical notion
%%%%%%%%%%%%%%%%%%%%%%%%%%%%%%%%%%%%%%%%%%%%%%%%%%%%%%%%%%%%%%%%%%%%%%
\begin{thm}\label{AGfqp1}
Let $(A, \m)$ be a local ring and $I$ a proper ideal of $A$. Then:
\begin{enumerate}
\item $A\bowtie I$ is arithmetical if and only if $A$ is arithmetical and $I=0$.
\item $A\bowtie I$ is Gaussian  if and only if  $A$ is Gaussian, $I^{2}=0$, and $aI=a^{2}I\ \forall\ a\in \m$.
\item $A\bowtie I$ is an fqp-ring  if and only if $A$ is an fqp-ring (resp., a Pr\"ufer ring), $\big(\Ze(A)\big)^{2}=0$, and $aI=a^{2}I\ \forall\ a\in \m$.
\end{enumerate}
\end{thm}

The proof of this theorem draws on the following results.

%%%%%%%%%%%%%%%%%%%%%%%%%%%%%%%%%%%%%%%%%%%%%%%%%%%%%%%%%%%%%%%%%%%%%%
%%%%%%%%%%%%%%%%%%%%%%%%%%%%%%%%%%%%%%%%%%%%%%%%%%%%%%%%%%%%%%%%%%%%%%
\begin{lemma}\label{fqp2}
Let $A$ be a ring and $I$ a proper ideal of $A$. Let $J$ be an ideal of $A$ and $K$ a subideal of $I$. Then $J\bowtie K$ is an ideal of $A\bowtie I$ if and only if $JI\subseteq K$.
\end{lemma}

%%%%%%%%%%%%%%%%%
%%%%%%%%%%%%%%%%%
\begin{proof}
The proof is straightforward and may be left to the reader.
\end{proof}

%%%%%%%%%%%%%%%%%%%%%%%%%%%%%%%%%%%%%%%%%%%%%%%%%%%%%%%%%%%%%%%%%%%%%%
%%%%%%%%%%%%%%%%%%%%%%%%%%%%%%%%%%%%%%%%%%%%%%%%%%%%%%%%%%%%%%%%%%%%%%
\begin{lemma}\label{fqp3}
Let $A$ be a ring and $I$ a proper ideal of $A$. If $A\bowtie I$ is an fqp-ring, then $A$ is an fqp-ring.
\end{lemma}

%%%%%%%%%%%%%%%%%
%%%%%%%%%%%%%%%%%
\begin{proof}
Assume that $A\bowtie I$ is an fqp-ring and let $J:=(a_{1},...,a_{n})$ be a finitely generated ideal of $A$, $K$ a subideal of $J$, and $f\in \Hom_{A}(J,J/K)$. We need to prove the existence of $g\in \Hom_{A}(J,J)$ such that $f=\overline{g}$ (mod $K$). For this purpose, consider the ideal of $A\bowtie I$ given by $U:=J\bowtie JI$ (Lemma~\ref{fqp2}). We claim that $U=\big((a_{1},a_{1}),...,(a_{n},a_{n})\big)$. Obviously, $(a_{i},a_{i})\in U\ \forall\ i=1, ..., n$. Next, let $(x,x+h)\in U$. We have
\[\begin{array}{rcl}
(x,x+h)     &=      &(x,x)+(0,\sum e_jf_j)\ \big(\mbox{for some } x\in J \mbox{ and } (e_j,f_j)\in J\times I, 1\leq j\leq m\big)\\
            &=      &\sum_{i} (r_i,r_i)(a_i,a_i)+ (0,\sum_j(\sum_{i} s_{ij} a_i)f_j)\ \big(\mbox{for some } r_i, s_{ij} \in A, 1\leq i\leq n\big)\\
            &=      &\sum_{i} (r_i,r_i)(a_i,a_i)+\sum_j(\sum_{i} s_{ij}a_i,\sum s_{ij}a_i)(0,f_j)\\
            &=      &\sum_{i} (r_i,r_i)(a_i,a_i)+\sum_j(0,f_j)\sum_{i} (s_{ij},s_{ij})(a_i,a_i)\\
            &=      &\sum_{i} (r_i,r_i+\sum_{j}f_js_{ij})(a_i,a_i), \mbox{ as desired.}
\end{array}\]
Let $V:=K\bowtie KI$, a subideal of $U$ by Lemma~\ref{fqp2}, and consider the function
\[\begin{array}{cccc}
 F:     &U                                      &\longrightarrow        & U/V \cong J/K \bowtie JI/KI\\
        &\sum_{i=1}^{n} \lambda_{i}(a_i,a_i)    &\longrightarrow        &\sum_{i=1}^{n} \lambda_{i}\big(f(a_i),f(a_i)\big).
 \end{array}\]
One can check that $F$ is well-defined and hence an $A\bowtie I$-map. Since $U$ is quasi-projective, there exists $G\in \Hom_{A\bowtie I}(U,U)$ such that $F=\overline{G}$ (mod $V$). Now, let $a\in J$ and let $g(a)$ equal the first coordinate of $G(a,a)$. Clearly, $g\in \Hom_{A}(J,J)$. Moreover, $\overline{G(a,a)}=F(a,a)=(f(a),f(a))$ yields $f=\overline{g}$.
\end{proof}

%%%%%%%%%%%%%%%%%%%%%%%%%%%%%%%%%%%%%%%%%%%%%%%%%%%%%%%%%%%%%%%%%%%%%%
%%%%%%%%%%%%%%%%%%%%%%%%%%%%%%%%%%%%%%%%%%%%%%%%%%%%%%%%%%%%%%%%%%%%%%
\begin{lemma}[{\cite[Theorem 2]{SM}}]\label{fqp4}
Let $R$ be a local fqp-ring which is not a chained ring. Then $\big(\Nil(R)\big)^{2}=0$.
\end{lemma}

%%%%%%%%%%%%%%%%%%%%%%%%%%%%%%%%%%%%%%%%%%%%%%%%%%%%%%%%%%%%%%%%%%%%%%
%%%%%%%%%%%%%%%%%%%%%%%%%%%%%%%%%%%%%%%%%%%%%%%%%%%%%%%%%%%%%%%%%%%%%%
\begin{lemma}[{\cite[Lemma 4.5]{AJK}}]\label{fqp5}
Let $R$ be a local fqp-ring which is not a chained ring. Then $\Ze(R)=\Nil(R)$.
\end{lemma}

%%%%%%%%%%%%%%%%%%%%%%%%%%%%%%%%%%%%%%%%%%%%%%%%%%%%%%%%%%%%%%%%%%%%%%%%%%%%%%%%%%%%%%%%%%%%
\begin{proof}[Proof of Theorem~\ref{AGfqp1}]
(1)  Assume that $A\bowtie I$ is (local) arithmetical (i.e., chained ring). Then $A$, too, is (local) arithmetical since the arithmetical property is stable under factor rings. Moreover, for each $i\in I$, the ideals $(i,0)A\bowtie I$ and $(0,i)A\bowtie I$ are comparable. In case $(i,0)\in (0,i)A\bowtie I$, there is an element $(a,j)\in A\times I$ such that $(0,i)=(a,a+j)(i,0)=(ai,0)$; so that $i=0$. Similarly, the other case yields $i=0$. So, we conclude that $I=0$, as desired. The converse is trivial since $A\bowtie (0)\cong A$.

(2)  Assume $A\bowtie I$ is (local) Gaussian. Then so is $A$ since the Gaussian property is stable under factor rings. Next, we prove $I^{2}=0$. Let $a, b \in I$. In $A\bowtie I$, we have $((a,a), (0,a))^{2}=((0,a)^{2})$ or $((a,a)^{2})$. The two cases yield, respectively, $a^{2}=0$ or $a^{2}(1-i)=0$ for some $i\in I\subseteq \m$. It follows that $a^{2}=0$. Likewise, $b^{2}=0$. Now appeal to the Gaussian property in $A$ to get $ab=0$. To prove the last statement, let $a\in A$ and $i\in I$. In $A\bowtie I$, we have $((a,a), (0,i))^{2}=((a,a)^{2})$ since $I^{2}=0$. It follows that $ai=a^{2}j$ for some $j\in I$. That is, $aI=a^{2}I$, as desired.

Conversely, let $(a,a+i), (b,b+j)\in A\bowtie I$. Since $A$ is local Gaussian, then, say, $(a, b)^{2}=(a^{2})$. Hence $b^{2}=a^{2}x$ and $ab=a^{2}y$ for some $x,y\in A$. Moreover, $ab=0$ implies $b^{2}=0$. By assumption, there exist $i_{1},i_{2},i_{3},j_{1},j_{2}\in I$ such that $2bj=a^{2}xj_{1}$, $2axi=a^{2}i_{1}$, $aj=a^{2}j_{2}$, $bi=a^{2}xi_{2}$, and $2ayi=a^{2}i_{3}$. Using the fact $I^{2}=0$, simple calculations show that $(b,b+j)^{2}=(a,a+i)^{2}(x,x+xj_{1}-i_{1})$ and $(a,a+i)(b,b+j)=(a,a+i)^{2}(y,y+xi_{2}+j_{2}-i_{3})$. Further, assume $(a,a+i)(b,b+j)=0$. Hence $ab=0$, whence $b^{2}=0$ and $2bj=0$ since $bI=0$. So that $(b,b+j)^{2}=0$. Therefore, $A\bowtie I$ is (local) Gaussian, completing the proof of (2).

(3) Without loss of generality, we assume that $A\bowtie I$ is a local fqp-ring that is not a chained ring (i.e., $I\not= 0$). By Lemma~\ref{fqp3}, $A$ is an  fqp-ring (and hence a Pr\"ufer ring). So $\Ze(A)$ is a (prime) ideal of $A$. Moreover, by (2), $aI=a^{2}I$ for every $a\in A$. In particular,  $I=aI$ for every regular element $a$ of $A$ and, hence, $I\subseteq \Ze(A)$ by Remark~\ref{P5.1}. So, by Lemma~\ref{P4}, $\Ze(A\bowtie I)=\Ze(A)\bowtie I$. Finally, (1) combined with Lemmas \ref{fqp4} \& \ref{fqp5} yield $\big(\Ze(A\bowtie I)\big)^{2}=0$. Consequently, $\big(\Ze(A)\big)^{2}=0$.

Conversely, assume that $A$ is a Pr\"ufer ring, $\big(\Ze(A)\big)^{2}=0$, and $aI=a^{2}I$ for every element $a$ in $A$. We aim to prove that $A\bowtie I$ is an fqp-ring. Throughout the proof, we will also be using the basic facts $I\subseteq \Ze(A)$, $I^2=0$, and $I=aI\ \forall\ a\in A\setminus\Ze(A)$.

Let $(a,a+i)$ and $(b,b+j)$ be two nonzero \emph{incomparable} elements of $A\bowtie I$.

\begin{claim} $a,b \in \Ze(A)$ \end{claim}

\noindent Indeed, assume, by way of contradiction, that $a$ is regular in $A$. By Lemma~\ref{P3.2}, $(a)$ and $(b)$ are comparable. Suppose $b=ac$ for some $c\in A$. There is $k\in I\ (=aI)$ such that $j-ci=ak$ which leads to $(b,b+j)=(a,a+i)(c,c+k)$, absurd. Now, if $a=bc$ for some $c\in A$, necessarily, $b$ is regular and hence similar arguments lead to the same absurdity, proving the claim.

\begin{claim} $\Ann(a,a+i)=\Ann(b,b+j)$ and $\big((a,a+i)\big)\cap\big((b,b+j)\big)=0$.\end{claim}

\noindent Clearly, $\Ann(a,a+i)\subseteq \Ze(A\bowtie I)=\Ze(A)\bowtie I$ by Lemma~\ref{P4}. The reverse inclusion is straight in view of Claim 1. So $\Ann(a,a+i)=\Ze(A)\bowtie I=\Ann(b,b+j)$, as desired. It remains to show that $\big((a,a+i)\big)\cap\big((b,b+j)\big)=0$. For this purpose, let  $(x,x+h)$ and $(y,y+k)$  be two elements of $A\bowtie I$ such that  $$(a,a+i)(x,x+h)=(b,b+j)(y,y+k).$$ We get via Claim 1 $$ax=by\ \text{ and }\ xi=yj.$$ We claim that  $x$ or $y \in \Ze(A)$. Deny and assume, by way of contradiction, that both $x$ and $y$ are regular in $A$. By Lemma~\ref{P3.2}, $xA$ and $yA$ are comparable; say, $x=ry$ for some $r\in A$. So $ary=by$ and $ryi=yj$ yield $b=ra$ and $j=ri$. It follows that $(b,b+j)=(a,a+i)(r,r)$, the desired contradiction. Consequently,  $x$ or $y \in \Ze(A)$. This forces $ax=by=xi=yj=0$, completing the proof of the claim.

Finally, let $J$ be a finitely generated ideal of $A\bowtie I$ with a minimal generating set $\big\{(a_{1},a_{1}+i_{1}), \ldots, (a_{n},a_{n}+i_{n})\big\}$. By Claim 2, we obtain $$\Ann(a_{h},a_{h}+i_{h})=\Ann(a_{k},a_{k}+i_{k}),\ \forall\ h\not=k\in\{1, \ldots, n\};$$
$$J=\big((a_{1},a_{1}+i_{1})\big)\oplus \big((a_{2},a_{2}+i_{2})\big)\oplus \ldots \oplus\big((a_{n},a_{n}+i_{n})\big).$$
Therefore $\big((a_{h},a_{h}+i_{h})\big)\cong \big((a_{k},a_{k}+i_{k})\big)$ and so $\big((a_{h},a_{h}+i_{h})\big)$ is $\big((a_{k},a_{k}+i_{k})\big)$-projective, for all $h, k$. By \cite[Corollary 1.2]{FH}, $J$ is quasi-projective, completing the proof of the theorem.
\end{proof}

%%%%%%%%%%%%%%%%%%%%%%%%%%%%%%%%%%%%%%%%%%%%%%%%%%%%%%%%%%%%%%%%%%%%%%
%%%%%%%%%%%%%%%%%%%%%%%%%%%%%%%%%%%%%%%%%%%%%%%%%%%%%%%%%%%%%%%%%%%%%%
\begin{remark}\label{fqp6} It is worthwhile noting that, in Theorem~\ref{AGfqp1}(2-3), the two facts $I^{2}=0$ and $(\Ze(A))^{2}=0$ are independent of the assumption ``$aI=a^{2}I,\ \forall\ a\in A$." For instance, this latter does not hold for the chained ring $A$ and ideal $I$ given in Remark~\ref{P5.1}(2); though $I^{2}=(\Ze(A))^{2}=0$ since $I\subset\Ze(A)=0\ltimes \Q$. Conversely, let $A:=\frac{\Z}{8\Z}$ and $I:=\frac{4\Z}{8\Z}$. One can verify that $(\Ze(A))^{2}=I\not=0$ and $aI=a^{2}I=0,\ \forall\ a\in \frac{2\Z}{8\Z}$.
\end{remark}

The next corollary handles the global case.

%%%%%%%%%%%%%%%%%%%%%%%%%%%%%%%%%%%%%%%%%%%%%%%%%%%%%%%%%%%%%%%%%%%%%%
%%%%%%%%%%%%%%%%%%%%%%%%%%%%%%%%%%%%%%%%%%%%%%%%%%%%%%%%%%%%%%%%%%%%%%
\begin{corollary}\label{AGfqp2}
Let $A$ be a ring and $I$ a proper ideal of $A$. Then:
\begin{enumerate}
\item $A\bowtie I$ is arithmetical if and only if $A$ is arithmetical and $I_{\m}=0,\ \forall\ \m\in\Max(A,I)$.
\item $A\bowtie I$ is locally an fqp-ring if and only if $A$ is locally an fqp-ring, $\big(\Ze(A_{\m})\big)^{2}=0$, and $aI_{\m}=a^{2}I_{\m},\ \forall\ \m\in\Max(A,I)$ and $\forall\ a\in \m$.
\item $A\bowtie I$ is Gaussian if and only if $A$ is Gaussian, $I_{\m}^{2}=0$, and $aI_{\m}=a^{2}I_{\m},\ \forall\ \m\in\Max(A,I)$ and $\forall\ a\in \m$.
\end{enumerate}
\end{corollary}

%%%%%%%%%%%%%%%%%
%%%%%%%%%%%%%%%%%
\begin{proof}
Let $\m\in\Max(A)$. By Remark~\ref{P5}, $(A\bowtie I)_{\m\bowtie I}\cong A_{\m}\bowtie I_{\m}$ if $I\subseteq \m$ and $(A\bowtie I)_{\m\bowtie I}\cong A_{\m}$ if $I\nsubseteq \m$. So Theorem~\ref{AGfqp1} combined with the known facts that the arithmetical and Gaussian properties are local leads to the conclusion.
\end{proof}

As an application of Corollary~\ref{AGfqp2}, one can easily construct new examples of non-local arithmetical, fqp, or Gaussian rings as shown below. The ``non-local" assumption here is meant to discriminate against the family of local Pr\"ufer-like rings which can be built via \cite[Theorem 4.4]{AJK} or \cite[Theorem 3.1]{BKM}.

%%%%%%%%%%%%%%%%%%%%%%%%%%%%%%%%%%%%%%%%%%%%%%%%%%%%%%%%%%%%%%%%%%%%%%
%%%%%%%%%%%%%%%%%%%%%%%%%%%%%%%%%%%%%%%%%%%%%%%%%%%%%%%%%%%%%%%%%%%%%%
\begin{example}\label{fqp7} Let $A:=\frac{\Z}{12\Z}$, $\m_{1}:=2A$, $\m_{2}:=3A$, and $I:=\m_{1}\m_{2}$. Then $A\bowtie I=A\ltimes I$ is locally an fqp ring which is not arithmetical. Indeed, $\left(\Ze(A_{\m_{i}})\right)^{2}=\m_{i}^{2}A\m_{i}=0$ (for $i=1,2$); $I_{\m_{1}}=6A_{\m_{1}}\not=0$; $I_{\m_{2}}=0$; and readily  $aI_{\m_{1}}=a^{2}I_{\m_{1}}=0,\ \forall\ a\in\m_{1}$.
\end{example}

%%%%%%%%%%%%%%%%%%%%%%%%%%%%%%%%%%%%%%%%%%%%%%%%%%%%%%%%%%%%%%%%%%%%%%
%%%%%%%%%%%%%%%%%%%%%%%%%%%%%%%%%%%%%%%%%%%%%%%%%%%%%%%%%%%%%%%%%%%%%%
\begin{question}\label{fqp8}
Is there a satisfactory global analogue of Theorem~\ref{AGfqp1}(3) for the fqp property (i.e, $A$ not necessarily local)?
\end{question}

%%%%%%%%%%%%%%%%%%%%%%%%%%%%%%%%%%%%%%%%%%%%%%%%%%%%%%%%%%%%%%%%%%%%%%%%%%%%%%%%%%%%%%%%%%%%%%%%%%%%%%%%%%%%%%%%%%%%%%%%%%%%%%%%%%%%%%%%%%%%%%
%%%%%%%%%%%%%%%%%%%%%%%%%%%%%%%%%%%%%%%%%%%%%%%%%%%%%%%%%%%%%%%%%%%%%%%%%%%%%%%%%%%%%%%%%%%%%%%%%%%%%%%%%%%%%%%%%%%%%%%%%%%%%%%%%%%%%%%%%%%%%%
%%%%%%%%%%%%%%%%%%%%%%%%%%%%%%%%%%%%%%%%%%%%%%%%%%%%%%%%%%%%%%%%%%%%%%%%%%%%%%%%%%%%%%%%%%%%%%%%%%%%%%%%%%%%%%%%%%%%%%%%%%%%%%%%%%%%%%%%%%%%%%
%%%%%%%%%%%%%%%%%%%%%%%%%%%%%%%%%%%%%%%%%%%%%%%%%%%%%%%%%%%%%%%%%%%%%%%%%%%%%%%%%%%%%%%%%%%%%%%%%%%%%%%%%%%%%%%%%%%%%%%%%%%%%%%%%%%%%%%%%%%%%%
%%%%%%%%%%%%%%%%%%%%%%%%%%%%%%%%%%%%%%%%%%%%%%%%%%%%%%%%%%%%%%%%%%%%%%%%%%%%%%%%%%%%%%%%%%%%%%%%%%%%%%%%%%%%%%%%%%%%%%%%%%%%%%%%%%%%%%%%%%%%%%
\section{The weak global dimension and transfer of the semihereditary condition}\label{WS}

A ring $R$ is semihereditary if every finitely generated ideal of $R$ is projective \cite{CE}. Recall for convenience that a ring $R$ has weak global dimension at most 1 (denoted $\w.dim(R)\leq1$) if and only if every finitely generated ideal of $R$ is flat if and only if $R_{p}$ is a valuation domain for every prime ideal $p$ of $R$ \cite[Theorem 3.4]{BG}. Therefore, if $\w.dim(R)\leq1$, then $R$ is arithmetical. Also, if $R$ is semihereditary, then $\w.dim(R)\leq1$; and the converse holds if $R$ is coherent (i.e., every finitely generated ideal is finitely presented) \cite[Theorem 3.3]{BG}.

The main result of this section investigates the weak global dimension of an amalgamation and its possible inheritance of the semihereditary condition.

%%%%%%%%%%%%%%%%%%%%%%%%%%%%%%%%%%%%%%%%%%%%%%%%%%%%%%%%%%%%%%%%%%%%%%
%%%%%%%%%%%%%%%%%%%%%%%%%%%%%%%%%%%%%%%%%%%%%%%%%%%%%%%%%%%%%%%%%%%%%%
\begin{thm}\label{WS1}
Let $A$ be a ring and $I$ a proper ideal of $A$. Then:
\begin{enumerate}
\item $\w.dim(A\bowtie I)\leq 1$ if and only if $\w.dim(A)\leq 1$ and $I_{\m}=0,\ \forall\ \m\in\Max(A,I)$.
\item Assume $I$ is finitely generated. Then $A\bowtie I$ is semihereditary if and only if $A$ is semihereditary and $I_{\m}=0,\ \forall\ \m\in\Max(A,I)$.
\end{enumerate}
\end{thm}

The proof requires the following lemma which examines the transfer of coherence to amalgamations.

%%%%%%%%%%%%%%%%%%%%%%%%%%%%%%%%%%%%%%%%%%%%%%%%%%%%%%%%%%%%%%%%%%%%%%
%%%%%%%%%%%%%%%%%%%%%%%%%%%%%%%%%%%%%%%%%%%%%%%%%%%%%%%%%%%%%%%%%%%%%%
\begin{lemma}\label{WS2}
Let $A$ be a ring  and $I$ a proper ideal of $A$. If $A\bowtie I $ is a coherent ring, then so is $A$. The converse holds when $I$ is finitely generated.
\end{lemma}

%%%%%%%%%%%%%%%%%
%%%%%%%%%%%%%%%%%
\begin{proof}
If $A\bowtie I$ is coherent, then $A$ is coherent, by \cite[Theorem 4.1.5]{G1}, since $A$ is a retract of $A\bowtie I$ via the retraction $\psi : A\bowtie I\longrightarrow A$ defined by $\psi(a,a+i)=a$. Conversely, assume that $A$ is coherent and $I$ is finitely generated. Recall that $I\times 0$ is an ideal of $A\bowtie I$ with $\frac{A\bowtie I}{I\times 0}\cong A$ \cite[Remark 1(b)]{D}. We claim that $I\times 0$ is $A\bowtie I$-coherent. Indeed, let $H$ be a  finitely generated subideal of $I\times 0$. We will show that $H$ is finitely presented. Clearly, $H:=\sum_{i=1}^{n}A\bowtie I(a_{i},0)$, for some positive integer $n$ and $a_i \in I$. Consider the exact sequence of $A\bowtie I -$modules $$0\rightarrow \Ker(u)\rightarrow (A\bowtie I)^{n}\stackrel{u}\rightarrow H\rightarrow 0$$ where $u(r_{i},r_{i}+e_{i})_{1\leq i\leq n}=\sum_{i=1}^{n}(r_{i},r_{i}+e_{i})(a_{i},0)=(\sum_{i=1}^{n}r_{i}a_{i},0)$. So that
$\Ker(u)=\{(r_{i},r_{i}+e_{i})_{1\leq i\leq n}\in(A\bowtie I)^{n}\mid\sum_{i=1}^{n}r_{i}a_{i}=0\}$. Now, set $J :=\sum_{i=1}^{n}Ra_{i}$, a finitely generated subideal of $I$, and consider the exact sequence of $A$-modules $$0\rightarrow \Ker(v)\rightarrow A^{n}\stackrel{v}\rightarrow J\rightarrow 0$$ where $v(b_{i})_{1\leq i\leq n}=\sum_{i=1}^{n}b_{i}a_{i}$. So, under the $A\bowtie I -$module identification $(A\bowtie I)^{n}=A^{n}\bowtie I^{n}$, we have $\Ker(u) = \Ker( v)\bowtie I^{n}$. But $J$ is finitely presented since $A$ is coherent. Hence, $\Ker(v) $ is a finitely generated $A$-module. Whence $\Ker(u)$ is a finitely generated $A\bowtie I$-module (recall $I$ is finitely generated). It follows that $H$ is finitely presented and thus $I\times 0$ is $A\bowtie I$-coherent, as claimed. By \cite[Theorem 2.4.1(2)]{G1}, $A\bowtie I$ is coherent, proving the lemma.
\end{proof}

%%%%%%%%%%%%%%%%%%%%%%%%%%%%%%%%%%%%%%%%%%%%%%%%%%%%%%%%%%%%%%%%%%%%%%%%%%%%%%%%%%%%%%%%%%%%
\begin{proof}[Proof of Theorem~\ref{WS1}]
(1) If $I_{\m}=0$ for every $\m\in\Max(A,I)$, then Remark~\ref{P5} yields $(A\bowtie I)_{\tilde{\m}}\cong(A\bowtie I)_{\m\bowtie I}\cong A_{\m}, \ \forall\ \m\in\Max(A)$, where $\tilde{\m}:=\{(x+i,x)\mid x\in \m,\ i\in I\}$. This fact combined with Corollary~\ref{AGfqp2}(1) leads to the conclusion.

(2) Combine Lemma~\ref{WS2} with (1) and the known fact that a ring is semihereditary if and only if it is coherent and has weak global dimension at most 1  \cite[Theorem 3.3]{BG}.
\end{proof}

%%%%%%%%%%%%%%%%%%%%%%%%%%%%%%%%%%%%%%%%%%%%%%%%%%%%%%%%%%%%%%%%%%%%%%
%%%%%%%%%%%%%%%%%%%%%%%%%%%%%%%%%%%%%%%%%%%%%%%%%%%%%%%%%%%%%%%%%%%%%%
\begin{example}\label{WS3}
By Theorem~\ref{WS1}, $\frac{\Z}{12\Z}\bowtie \frac{4\Z}{12\Z}=\frac{\Z}{12\Z}\ltimes \frac{4\Z}{12\Z}$ is a semihereditary ring since $(\frac{4\Z}{12\Z})_{\frac{2\Z}{12\Z}}=0$. Notice, however, that the above results do not allow for a discrimination among the classes of arithmetical rings, rings with weak global dimension at most 1, and semihereditary rings.
\end{example}

Recall at this point that Osofsky in 1969 (resp., Glaz in 2005) proved that the weak global dimension of an arithmetical (resp., a coherent Gaussian) ring is 0, 1, or $\infty$ \cite{Os} (resp., \cite[Theorem 3.3]{G2}). One can use amalgamations to build new examples of non-arithmetical non-Coherent Gaussian rings with infinite weak global dimension, as shown in the next example.

%%%%%%%%%%%%%%%%%%%%%%%%%%%%%%%%%%%%%%%%%%%%%%%%%%%%%%%%%%%%%%%%%%%%%%
%%%%%%%%%%%%%%%%%%%%%%%%%%%%%%%%%%%%%%%%%%%%%%%%%%%%%%%%%%%%%%%%%%%%%%
\begin{example}\label{WS4}
Let $A$ be a local non-coherent Gaussian ring and $0\not=I$ a proper ideal of $A$ with $I^{2}=0$ and $aI=a^{2}I\ \forall\ a\in A$. Assume $0\times I$ is not flat in $A\bowtie I$ (in particular, if $I$ is finitely generated or not flat in $A$). Then the amalgamation $R:=A\ltimes I$ is a local non-arithmetical non-coherent Gaussian ring with $\w.dim(R)=\infty$. For an explicit example, one may take $A:=\Z_{(2)}\ltimes\Q$ and $I:=0\ltimes\Q$.
\end{example}

%%%%%%%%%%%%%%%%%
%%%%%%%%%%%%%%%%%
\begin{proof}
By Theorem~\ref{AGfqp1}, Theorem~\ref{WS1}, and Lemma~\ref{WS2}, $R$ is a local non-arithmetical non-coherent Gaussian ring with $\w.dim(R)\geq2$. Next, assume $0\times I$ is not flat in $R$. Let $\{f_{i}\}_{i\in \Delta}$ be a  set of generators of $I$ and consider the $R$-map $u: R^{(\Delta)}\rightarrow 0\times I$ defined by $u(a_{i},e_{i})_{i\in \Delta}=\sum_{i\in \Delta} (a_{i},e_{i})(0,f_{i})=(0,\sum_{i\in \Delta}a_{i}f_{i})$. Clearly, $\Ker(u)=V\bowtie I^{(\Delta)}$, where $V:=\{(a_{i})_{i}\in A^{(\Delta)}/\sum_{i\in \Delta}a_{i}f_{i}=0 \}$. Here we are identifying $R^{(\Delta)}$ with $A^{(\Delta)}\bowtie I^{(\Delta)}$ as $R$-modules. We have the exact sequence of $R$ modules $$0\rightarrow V\bowtie I^{(\Delta)}\rightarrow R^{(\Delta)}\stackrel{u}\rightarrow 0\times I\rightarrow 0.$$
On the other hand, $V\bowtie I^{(\Delta)}=V^{\star}\oplus (0\times I)^{(\Delta)}$, where $V^{\star}=\{(a,a)/ r\in V\}$. Since $0\times I$ is not flat, the above sequence yields
$$\fd(0\times I)\leq \fd\left(V^{\star}\oplus (0\times I)^{(\Delta)}\right)\leq \fd(0\times I)-1.$$ Therefore $\fd(0\times I)=\w.dim(R)=\infty$, as desired.

Now, if $I$ is finitely generated, then $0\times I$ is not $R$-flat since $R$ is local and $(a,0)(0\times I)= 0$ for any $0\not=a\in I$. Also, using the interpretation of flatness in terms of relations \cite[Ch. I, \S 2, Corollary 1]{Bour}, one can easily verify that if $0\times I$ is $R$-flat, then $I$ is $A$-flat.

For the explicit example, it is readily seen that $I^{2}=0$ and $aI=a^{2}I\ \forall\ a\in 2\Z_{(2)}\ltimes\Q$. Moreover, $A$ is an arithmetical (hence Gaussian) ring by \cite[Theorem 2.1]{BKM} and not coherent by \cite[Theorem 2.8]{KM}. Finally, we claim that $I:=0\ltimes\Q$ is not flat in $A:=\Z_{(2)}\ltimes\Q$ since it is not flat in $\Q\ltimes\Q$. The evidence for this fact is already included in the proof of \cite[Lemma 2.3]{BKM}.
\end{proof}

%%%%%%%%%%%%%%%%%%%%%%%%%%%%%%%%%%%%%%%%%%%%%%%%%%%%%%%%%%%%%%%%%%%%%%
%%%%%%%%%%%%%%%%%%%%%%%%%%%%%%%%%%%%%%%%%%%%%%%%%%%%%%%%%%%%%%%%%%%%%%
\begin{question}
Is Theorem~\ref{WS1}(2) true if $I$ is not assumed to be finitely generated?
\end{question}

%%%%%%%%%%%%%%%%%%%%%%%%%%%%%%%%%%%%%%%%%%%%%%%%%%%%%%%%%%%%%%%%%%%%%%%%%%%%%%%%%%%%%%%%%%%%%%%%%%%%%%%%%%%%%%%%%%%%%%%%%%%%%%%%%%%%%%%%%%%%%%
%%%%%%%%%%%%%%%%%%%%%%%%%%%%%%%%%%%%%%%%%%%%%%%%%%%%%%%%%%%%%%%%%%%%%%%%%%%%%%%%%%%%%%%%%%%%%%%%%%%%%%%%%%%%%%%%%%%%%%%%%%%%%%%%%%%%%%%%%%%%%%
%%%%%%%%%%%%%%%%%%%%%%%%%%%%%%%%%%%%%%%%%%%%%%%%%%%%%%%%%%%%%%%%%%%%%%%%%%%%%%%%%%%%%%%%%%%%%%%%%%%%%%%%%%%%%%%%%%%%%%%%%%%%%%%%%%%%%%%%%%%%%%
%%%%%%%%%%%%%%%%%%%%%%%%%%%%%%%%%%%%%%%%%%%%%%%%%%%%%%%%%%%%%%%%%%%%%%%%%%%%%%%%%%%%%%%%%%%%%%%%%%%%%%%%%%%%%%%%%%%%%%%%%%%%%%%%%%%%%%%%%%%%%%
%%%%%%%%%%%%%%%%%%%%%%%%%%%%%%%%%%%%%%%%%%%%%%%%%%%%%%%%%%%%%%%%%%%%%%%%%%%%%%%%%%%%%%%%%%%%%%%%%%%%%%%%%%%%%%%%%%%%%%%%%%%%%%%%%%%%%%%%%%%%%%

\end{document}